\documentclass{amsart}
\usepackage{amssymb}
\usepackage{amsmath}
\usepackage{amsthm}
\usepackage{hyperref}
\numberwithin{equation}{section}
\newtheorem{theorem}{Theorem}[section]
\newtheorem{conjecture}[theorem]{Conjecture}

\newtheorem{lemma}[theorem]{Lemma}
\newtheorem{proposition}[theorem]{Proposition}
\newtheorem{corollary}[theorem]{Corollary}

\theoremstyle{remark}
\newtheorem{definition}[theorem]{Definition}
\newtheorem{remark}[theorem]{Remark}

\newcommand\co{\colon\,}

\newcommand\vol{\text{\textup{vol}}}
\newcommand\dom{\operatorname{dom}}

\newcommand\lp{\textup{(}}
\newcommand\rp{\textup{)}}
\newcommand\wt{\widetilde}
\newcommand\bZ{\mathbb Z}
\newcommand\bR{\mathbb R}
\newcommand\bN{\mathbb N}
\newcommand\bQ{\mathbb Q}
\newcommand\bT{\mathbb T}
\newcommand\bC{\mathbb C}

\newcommand\cL{\mathcal L}
\newcommand\cA{\mathcal B}
\newcommand\cS{\mathcal S}

\newcommand\bb{$\bullet$}
\newcommand\Ca{$C^*$-algebra}

\renewcommand\Im{\operatorname{Im}}
\newcommand{\dbar}{\overline{\partial}}
\newcommand{\nbar}{\overline{\nabla}}
\newcommand{\ira}{irrational rotation algebra}
\newcommand{\EL}{Eu\-ler-La\-grange}

\begin{document}
\title{Noncommutative variations on Laplace's equation}
\begin{abstract}
As a first step at developing a theory of noncommutative
nonlinear elliptic partial differential equations, we analyze
noncommutative analogues of Laplace's equation and its
variants (some of the them nonlinear) over noncommutative
tori. Along the way we prove noncommutative analogues of
many results in classical analysis, such as Wiener's Theorem
on functions with absolutely convergent Fourier series,
and standard existence and non-existence theorems on elliptic functions.
We show that many many classical methods, including the Maximum
Principle, the direct method of the calculus of variations, 
the use of the Leray-Schauder Theorem, etc.,
have analogues in the noncommutative setting.
\end{abstract}
\author{Jonathan Rosenberg}
\thanks{Partially supported by NSF Grants DMS-0504212 and
DMS-0805003. This paper 
grew out of work on the paper \cite{Noncommsigma}. Some
of the results of this paper were presented in the
Special Session on E-Theory, Extensions, and Elliptic Operators at
the Joint Mathematics Meetings, San Diego, California, January 9, 2008.}
\address{Department of Mathematics\\
University of Maryland\\
College Park, MD 20742, USA}
\email{jmr@math.umd.edu}
\urladdr{http://www.math.umd.edu/\raisebox{-.6ex}{\symbol{"7E}}jmr}
\maketitle

\section{Introduction}
\label{sec:intro}

Gelfand's Theorem shows that $X\rightsquigarrow 
C_0(X)$ sets a contravariant equivalence of categories from the
category of locally compact [Hausdorff] spaces and proper maps to
the category of commutative {\Ca}s and $*$-homomorphisms. This
observation is the key to the whole subject of noncommutative geometry,
which is based on the following dictionary:

\medskip
\begin{center}\begin{tabular}{||cc|c||} 
\hline
&{\textbf{Classical}}& \textbf{Noncommutative}\\
\hline
\bb&locally compact space & {\Ca}\\
\bb&compact space & unital  {\Ca}\\
\bb&vector bundle & f.~g.\ projective module\\
\raisebox{-5pt}{\bb}&\raisebox{-5pt}{smooth manifold} & {\Ca} with\\
&& ``smooth subalgebra''\\
\bb&real-valued function & self-adjoint element\\
\bb&partial derivative & unbounded derivation\\
\bb&integral & tracial state\\
\hline
\end{tabular}\end{center}
\medskip

The object of this paper is to begin to use this dictionary to set up
a noncommutative theory of elliptic partial differential equations,
both linear and nonlinear, along with corresponding aspects
of the calculus of variations. Since the theory is still in its infancy,
we begin with the very simplest case: Laplace's equation and PDEs
closely connected to it, and concentrate 
on the simplest nontrivial example of a
noncommutative manifold, the irrational rotation algebra
(or noncommutative $2$-torus) $A_\theta$, $\theta\in\bR\smallsetminus
\bQ$. A definition of elliptic partial differential operators,
along with the study of one example associated with the
irrational rotation algebra, was given in Connes' fundamental
paper \cite{MR572645}, but there the emphasis was on pseudodifferential
calculus and index theory.  Here we focus on several other things:
variational methods, the Maximum Principle, an analogue of
Wiener's Theorem, tools for treating \emph{nonlinear} equations,
the beginnings of a theory of harmonic unitaries, and some aspects of
noncommutative complex analysis.

What is the motivation for a noncommutative theory of elliptic PDE?
For the most part, it comes from physics. Many of the
classical elliptic PDEs arise from variational problems in
Riemannian geometry, and are also the field equations of physical
theories. 
But the \emph{uncertainty principle} forces quantum
observables to be noncommutative. 
There is also increasing evidence (e.g.,
\cite{MR1128127,MR1463819,MR1613978,MR1720697,MR2116734,MR2222224})
that quantum field theories should allow for the possibility of
noncommutative space-times.
\emph{Noncommutative sigma-models}, for which the very earliest
and simplest investigations are in \cite{MR1818819,MR2035840}, 
will require the noncommutative harmonic map equation, which generalizes
the Laplace equation studied in this paper.

We use as our starting point the ``noncommutative differential geometry''
of Alain Connes \cite{MR572645}. This theory only works well with
``highly symmetric'' noncommutative spaces, as the ``smooth'' elements
are taken to be the $C^\infty$ vectors for an action of a Lie group
on a {\Ca}, but this theory is well adapted to the case of the
{\ira}, which carries an ergodic ``gauge action'' of the $2$-torus
$\bT^2$.

The outline of this paper is as follows. We begin in Section
\ref{sec:lin} with the basic properties of the Laplacian on
$A_\theta$. Included are analogues of Wiener's Theorem (Theorem
\ref{thm:Wiener}) and the Maximum Principle (Proposition
\ref{prop:MaxPrin}). In Section \ref{sec:Sobolev}, we take up the basic
properties of Sobolev spaces on $A_\theta$, which are needed for a
deeper analysis of some aspects of noncommutative PDEs. We should
point out that some of the material of this section has already
appeared in \cite[\S3]{MR2247860} and in \cite{MR2277208}. 
The heart of this paper is
contained in Sections \ref{sec:nonlin} and \ref{sec:unitaries}, which
begin to develop a theory 
of \emph{nonlinear} elliptic partial differential equations, using
methods analogous to those traditional in the theory of nonlinear
elliptic PDE. Finally, Section \ref{sec:holo} deals with
noncommutative complex analysis.

We should mention that another example of noncommutative elliptic PDE
and an associated variational problem on noncommutative tori, namely,
noncommutative  Yang-Mills theory, has already been studied by Connes
and Rieffel \cite{MR878383,MR1037414}. Furthermore, Theorem
\ref{thm:Wiener} was previously proved by Gr{\"o}chenig
and Leinert \cite{MR2015328} by another method, and variations on the
Gr{\"o}chenig-Leinert work can be found in \cite{MR2277208}. In their
paper, Gr{\"o}chenig 
and Leinert point out some applications to harmonic analysis and
wavelet theory, which go
off in a somewhat different direction than the applications to
mathematical physics which we envisage, though obviously there is some
overlap between the two.

I would like the thank the referee for several useful comments and
especially for the reference to \cite{MR2015328}. I would also
like to thank Hanfeng Li for pointing out an error in the original
proof of Theorem \ref{thm:Wiener}.

\section{The linear Laplacian}
\label{sec:lin}

We will be studying the {\Ca}
$A_\theta$ generated by two unitaries $U$, $V$ satisfying $UV=e^{2\pi
  i\theta}VU$. $A_\theta$ is simple with unique trace $\tau$
if $\theta\in \bR\smallsetminus\bQ$. (See for example \cite{MR623572} for
a review of the basic facts about $A_\theta$.) The torus $G=\bT^2$ acts by 
\[\alpha_{(z_1,z_2)} U = z_1U,\quad \alpha_{(z_1,z_2)} V = z_2V,
\quad |z_1|=|z_2|=1.
\]
The space of $C^\infty$ vectors for the action $\alpha$ is the ``smooth
{\ira}''
\[
A_\theta^\infty = \left\{ \sum_{m,n} c_{m,n}U^mV^n :
c_{m,n} \text{ rapidly decreasing}\right\}.
\]
This should be viewed as a noncommutative deformation of the
algebra $C^\infty(\bT^2)$ of smooth functions on an ordinary
$2$-torus, and the decomposition of an element of this algebra
in terms of multiples of $U^mV^n$ should be viewed as a sort
of noncommutative Fourier series decomposition, with $c_{m,n}$ as
a sort of Fourier coefficient.  For $a\in A_\theta$ but not
necessarily in $A_\theta^\infty$, the Fourier coefficients
$c_{m,n}$ are well defined and satisfy $|c_{m,n}|\le \Vert a\Vert$, since
$c_{m,n} = \tau(V^{-n}U^{-m}a)$, but the Fourier series expansion
of $a$ is only a formal expansion, and need not converge in the
topology of $A_\theta$, just as one has functions in
$C(\bT^2)$ whose Fourier series do not converge absolutely or
even pointwise.

We denote by $\delta_1$ and $\delta_2$ the infinitesimal generators
of the actions of the two $\bT$ factors in $\bT^2$ under $\alpha$. These are
unbounded derivations on $A_\theta$, and map $A_\theta^\infty$ to
itself. They are given by
\[
\delta_1(U) = 2\pi i U, \quad \delta_2(V) = 2\pi i V, \quad
\delta_2(U) = \delta_1(V) = 0.
\]
These derivations $\delta_j$ obviously commute with the adjoint
operation $*$, and play the roles of the partial derivatives
$\partial/\partial x_j$ in classical analysis on the $2$-torus.
Since the action $\alpha$ of $\bT^2$ preserves the tracial state $\tau$,
$\tau\circ \delta_j = 0$, $j=1,\,2$. This fact is the basis for the
following Lemma, which we will use many times in the future.
\begin{lemma}[``Integration by Parts'']
\label{lem:intbyparts}
If $a,\,b\in A^\infty_\theta$, then $\tau(\delta_j(a)b) = -
\tau(\delta_j(b)a)$, $j=1,\,2$.
\end{lemma}
\begin{proof}
We have
\[
0 = \tau(\delta_j(ab)) = \tau(\delta_j(a) b) + \tau(a \delta_j(b)).
\]
The result follows.
\end{proof}

\begin{definition}
\label{def:Laplacian}
In analogy with the usual notation in analysis, we let
\[
\Delta = \delta_1^2 + \delta_2^2.
\]
This should be viewed as a ``noncommutative elliptic partial
differential operator.''  (The notion of ellipticity was defined
rigorously in \cite[p.\ 602]{MR572645}.)
Clearly, $\Delta$ is a ``negative'' operator, and its spectrum
consists of the numbers $-4\pi^2 (m^2 + n^2)$, $m,\,n\in \bZ$,
with eigenfunctions $U^mV^n$. Via the noncommutative Fourier
expansion discussed earlier, the pair $(A_\theta^\infty, \Delta)$
is isomorphic to $C^\infty(\bT^2)$ with the usual Laplacian $\Delta$,
provided one looks just at the linear structure and forgets the
noncommutativity of the multiplication. (This was already observed in
\cite[p.\ 602]{MR572645}.) 
\end{definition} 
\begin{proposition}
\label{prop:invLap}
For any $\lambda>0$ {\lp}or not of the form $-4\pi^2 n$ with
$n\in\bN${\rp}, $-\Delta +\lambda\co A_\theta^\infty \to
A_\theta^\infty$ is bijective.
\end{proposition}
\begin{proof}
We have 
\[
(-\Delta+\lambda)\left(\sum_{m,n} c_{m,n}U^mV^n\right) = \sum_{m,n}
\bigl(4\pi^2(m^2+n^2)+\lambda\bigr) \,c_{m,n}U^mV^n\,.
\]
It is immediate that $-\Delta +\lambda$ has no kernel and has an
inverse given by the formula
\[
\sum_{m,n} c_{m,n}U^mV^n \mapsto \sum_{m,n} 
\frac{1}{4\pi^2(m^2+n^2)+\lambda}\,c_{m,n}U^mV^n\,,
\]
since if $c_{m,n}$ is rapidly decreasing, so are the coefficients
on the right.
\end{proof}
It is also easy to characterize the image of $\Delta$.
\begin{proposition}
\label{prop:imLap}
The image of $\Delta\co A_\theta^\infty \to A_\theta^\infty$ is
precisely $A_\theta^\infty \cap \ker \tau$, the smooth elements with
zero trace.
\end{proposition}
\begin{proof}
We have $\Delta(\sum_{m,n} c_{m,n}U^mV^n) = -4\pi^2\sum_{m,n}
(m^2+n^2) c_{m,n}U^mV^n$, and the factor $(m^2+n^2)$ kills the term
with $m=n=0$. Thus the image of $\Delta$ is contained in
the kernel of $\tau$. Conversely, suppose $a=\sum_{m,n} d_{m,n}U^mV^n$ is an
arbitrary element of $A_\theta^\infty \cap \ker \tau$. That means
$d_{m,n}$ is rapidly decreasing and $d_{0,0}=0$. Then
$d_{m,n}/(m^2+n^2)$ is also rapidly decreasing, and
\[
\sum_{m,n}{}^\prime \frac{- d_{m,n}}{4\pi^2\,(m^2+n^2)} U^mV^n,
\]
where the $'$ indicates we omit the term with $m=n=0$,
converges to an element $b$ of $A_\theta^\infty$ with $\Delta b = a$.
\end{proof}
The following consequence is an analogue of a well-known fact about
subharmonic functions on compact manifolds.
\begin{corollary}
\label{cor:nosubh}
If $a\in A_\theta^\infty$ is subharmonic {\lp}i.e., if $\Delta a\ge
0${\rp}, then $a$ is constant.
\end{corollary}
\begin{proof}
Suppose $a\in A_\theta^\infty$ and $\Delta a\ge 0$. By Proposition
\ref{prop:imLap}, $\tau(\Delta a)=0$. But $\tau$ is a faithful trace,
which means that if $b\ge 0$ and $\tau(b)=0$, then $b=0$. Apply this
with $b=\Delta a$ and we see that $\Delta a=0$. This implies $a$ is a
scalar multiple of $1$.
\end{proof}

For future use, we are also going to want to study other ``function
spaces'' on the noncommutative torus. For example, we have the analogue
of the Fourier algebra of functions with absolutely convergent Fourier
series.
\begin{definition}
\label{def:Fourieralg}
Fix $\theta\in\bR\smallsetminus\bQ$, and let
\[
\cA_\theta = \left\{ \sum_{m,n} c_{m,n}U^mV^n :
\sum_{m,n} \left\vert c_{m,n} \right\vert < \infty\right\}.
\]
This is obviously a Banach subspace of $A_\theta$ with norm $\Vert\,\cdot
\,\Vert_{\ell^1}$ given
by the $\ell^1$ norm of the coefficients $c_{m,n}$. We also obviously
have $\Vert a\Vert_{\ell^1}\ge \Vert a\Vert$ for $a\in\cA_\theta$.
($\Vert \,\cdot\,\Vert$ will for us always denote the {\Ca} norm.)
\end{definition}
The following lemma, related in spirit to the
Sobolev Embedding Theorem \cite[Theorem 1.1]{Kazdan}, relates the
topology of $\cA_\theta$ to the subject of Propositions
\ref{prop:invLap} and \ref{prop:imLap}.  More details of
noncommutative Sobolev space theory will be taken up in Section 
\ref{sec:Sobolev} below.
\begin{lemma}
\label{lem:L1fromD}
Let $f\in A_\theta^\infty$. Then there is a constant $C>0$ such
that {\lp}in the notation of Definition \textup{\ref{def:Fourieralg})}
$\Vert f\Vert_{\ell^1} \le C \Vert (-\Delta+1)f\Vert$. In particular,
the domain of $\Delta$, as an unbounded operator on $A_\theta$,
is contained in $\cA_\theta$.
\end{lemma}
\begin{proof}
Suppose $f=\sum_{m,n} c_{m,n} U^mV^n \in A_\theta^\infty$. Then
\[
\Vert f\Vert_{\ell^1} = \sum_{m,n} |c_{m,n}| =
\sum_{m,n} \bigl(1+4\pi^2(m^2+n^2)\bigr)\, 
c_{m,n}\cdot \frac{a_{m,n}}{1+4\pi^2(m^2+n^2)},
\]
where $|a_{m,n}|=1$. View this as an $\ell^2$ inner product and
estimate it by Cauchy-Schwarz. We obtain
\[
\Vert f\Vert_{\ell^1} \le  C\Vert (-1+\Delta)f\Vert_{\ell^2},
\]
where $\Vert\, \cdot \,\Vert_{\ell^2}$ is the $\ell^2$ norm of the
sequence of Fourier coefficients (this can also be defined by
$\Vert c\Vert_{\ell^2}=\tau(c^*c)^{\frac12}$) and where
\[
C = \left\Vert \{\bigl(1+4\pi^2(m^2+n^2)\bigr)^{-1}\}_{m,n}
\right\Vert_{\ell^2} = \left( \sum_{m,n} 
\frac{1}{\bigl(1+4\pi^2(m^2+n^2)\bigr)^2} \right)^{\frac12} <
\infty.
\]
Since the $\ell^2$ norm is bounded by the {\Ca} norm, as
$\Vert c\Vert_{\ell^2}=\tau(c^*c)^{\frac12} \le \Vert c^*c\Vert^{\frac12}
= \Vert c\Vert$, the result follows.
\end{proof}

The following result was proved several years ago by Gr{\"o}chenig
and Leinert \cite{MR2015328}, using the theory of symmetric
$L^1$-algebras as developed by Leptin, Ludwig, Hulanicki, \emph{et al}. We
include a brief proof here for the sake of completeness.
\begin{theorem}[``Wiener's Theorem'']
\label{thm:Wiener}
The Banach space $\cA_\theta$ is a Banach $*$-algebra and is
closed under the holomorphic functional calculus of $A_\theta$.
Thus if $a\in \cA_\theta$ and $a$ is invertible in $A_\theta$,
$a^{-1}\in \cA_\theta$.
\end{theorem}
\begin{proof}
Suppose $a=\sum c_{m,n}U^mV^n$ with the sum absolutely convergent.
Then
\[
a^* = \sum_{m,n} \overline{c_{m,n}}V^{-n}U^{-m} =
\sum_{m,n} \overline{c_{m,n}}e^{-2\pi i mn\theta}U^{-m}V^{-n}
\]
so $a^*\in \cA_\theta$. Similarly, if also $b=\sum d_{m,n}U^mV^n$ 
(absolutely convergent sum), then $ab$ has Fourier coefficients
given by ``twisted convolution'' of the Fourier coefficients of $a$ and $b$,
and since the twisting only involves scalars of absolute value $1$,
the Fourier coefficients of $ab$ are absolutely convergent. More precisely,
\[
\begin{aligned}
ab & = \left(\sum_{m,n} c_{m,n}U^m V^n\right)\left( 
\sum_{k,l} d_{k,l}U^k V^l\right)\\
&=\sum_{m,n,k,l} c_{m,n} d_{k,l} U^m V^n U^k V^l\\
&=\sum_{m,n,k,l} c_{m,n} d_{k,l} e^{-2\pi i k n \theta} U^{m+k} V^{n+l}\\
&=\sum_{p,q} f_{p,q} U^p V^q,\quad\text{where}\\
f_{p,q} &= \sum_{m,n} c_{m,n} d_{p-m,q-n} e^{-2\pi i (p-m) n \theta},
\text{ so that}\\
|f_{p,q}| &\le \sum_{m,n} |c_{m,n}|\, |d_{p-m,q-n}| \le \Vert c\Vert_{\ell^1}
\Vert d\Vert_{\ell^1}.
\end{aligned}
\]
This confirms that $\cA_\theta$ is a Banach $*$-algebra and of course
a $*$-subalgebra of $A_\theta$.

To prove the analogue of Wiener's Theorem, we unfortunately cannot
use the cute proof using the Gelfand transform, since $\cA_\theta$ is
not commutative. We also cannot use another very elementary proof from
\cite{MR0365002} since this also relies on commutativity.  However
Newman's proof is related to another well-known fact (implicit
in \cite[Lemma 1]{MR572645}), that $A_\theta^\infty$ is
closed under the holomorphic functional calculus of $A_\theta$.
To prove this one has to show that if $b\in A_\theta^\infty$ with 
$b$ invertible in $A_\theta$, then $b^{-1}$ also lies in
$A_\theta^\infty$. To prove this fact, iterate the identity
$\delta_j(b^{-1})= -b^{-1}\,\delta_j(b)\,b^{-1}$
to see that $b^{-1}$ lies in the domain of all monomials in $\delta_1$
and $\delta_2$. One might think that since $A_\theta^\infty$ is dense
in $\cA_\theta$, this should be enough to prove Wiener's Theorem for
the latter, but this doesn't work, since in general the spectrum and
spectral radius functions are only upper semicontinuous, not
continuous, on a noncommutative Banach algebra \cite{MR0051441}.

To prove the theorem, we rely on an observation of Hulanicki
\cite[Proposition 2.5]{MR0323951} based on a theorem
of Raikov \cite[Theorem 5]{MR0019845}: that if
a Banach $*$-algebra $B$ (with isometric involution and a faithful
$*$-representation on a Hilbert space) is embedded in
its enveloping {\Ca} $A$, then the spectra of self-adjoint elements of
$B$ are the same whether computed in $B$ or in $A$ if and only if $B$
is symmetric (i.e., for $x\in B$, the spectrum in $B$ of $x^*x$ is
contained in $[0,\infty)$). We will apply this with $B=\cA_\theta$ and
with $A=A_\theta$. Hulanicki also showed \cite{MR0278082}
that the $L^1$ algebras of discrete nilpotent groups are symmetric.
In particular, the $L^1$ algebra of the discrete Heisenberg group $H$
(with generators $a$, $b$, $c$, where $c$ is central and
$aba^{-1}b^{-1}=c$) is symmetric. Thus $\cA_\theta$, which is the
quotient of $L^1(H)$ by the (self-adjoint)
ideal generated by $c-e^{2\pi i\theta}$,
is also symmetric. (If $B$ is a symmetric Banach $*$-algebra and $J$
is a closed self-adjoint ideal, then $B/J$ is also symmetric,
since if $\dot x\in B/J$ is the image of $x\in B$, then the
spectrum of $\dot x^*\dot x$ in $B/J$ is contained in the spectrum
of $x^*x$ in $B$, hence is contained in $[0,\infty)$.)
So for $x=x^*\in \cA_\theta$, by Hulanicki's
observation, if $x$ is invertible in $A_\theta$, $x^{-1}\in
\cA_\theta$. Suppose $a\in \cA_\theta$ and $a$ is invertible in
$A_\theta$. Then $a^*$ is also invertible in $A_\theta$, so $x=a^*a\in
\cA_\theta$ and $x$ is invertible in $A_\theta$. Hence
$x^{-1}=a^{-1}{a^*}^{-1}\in \cA_\theta$ and $a^{-1}=x^{-1}a^* \in \cA_\theta$.
\end{proof}

In the classical theory of the Laplacian, one of the most useful
tools is the ``Maximum Principle'' (e.g., \cite[p.\ 20]{Kazdan}).
The following is a noncommutative
analogue.
\begin{proposition}[``Maximum Principle'']
\label{prop:MaxPrin}
Let $h=h^*\in A_\theta^\infty$, and let $[t_0,t_1]$ be the smallest
closed interval containing the spectrum $\sigma(h)$ of $h$ 
{\lp}in $A_\theta${\rp}. In other words, let $t_1 = \max \{t:t\in\sigma(h)\}$
and $t_0 = \min \{t:t\in\sigma(h)\}$. Then there exists a state $\varphi$
of $A_\theta$ with $\varphi(h)=t_1$, and for such a state,
$\varphi(\Delta h) \le 0$. Similarly, there exists a state $\psi$
of $A_\theta$ with $\psi(h)=t_0$, and for such a state,
$\psi(\Delta h) \ge 0$. 
\end{proposition}
\begin{proof}
The commutative {\Ca} $C^*(h)$ must have pure states $\wt\varphi$ and
$\wt\psi$ with $\wt\varphi(h)=t_1$, $\wt\psi(h)=t_0$, since
$t_0,\,t_1\in\sigma(h)$. 
Extend these to states $\varphi$, $\psi$ of the larger {\Ca} $A_\theta$.
Then for $s\in G=\bT^2$, the functions $s\mapsto \varphi(\alpha_s(h))$
and $s\mapsto \psi(\alpha_s(h))$ must have a maximum (resp., minimum)
at the identity element of $\bT^2$. (Recall that $\alpha$ is the
gauge action by $*$-automorphisms.) Differentiate twice and the
result follows by the ``second derivative test.''
\end{proof}

Just as in the classical setting, Laplace's equation arises as the
{\EL} equation of a variational problem. 
\begin{definition}
\label{def:energy1}
For $a\in A_\theta^\infty$, let 
\[
E(a) = \frac12\tau\bigl(\delta_1(a)^2 + \delta_2(a)^2\bigr).
\]
This is
clearly the noncommutative analogue of the classical energy functional
\[
f \mapsto  \frac12\int_M \Vert \nabla f \Vert^2\,d\vol
\]
on a compact manifold $M$.
\end{definition}
\begin{proposition}
\label{prop:classEL}
The {\EL} equation for critical points of the energy functional $E$ of
Definition \textup{\ref{def:energy1}}, restricted to self-adjoint elements of
$A_\theta^\infty$,  is just Laplace's equation $\Delta a = 0$. Thus
the only critical points are the scalar multiples of the identity,
which are the points where $E(a)=0$ and are strict minima for $E$.
\end{proposition}
\begin{proof}
This works very much like the classical case. If $a=a^*$ and $h=h^*$,
then
\[
\left.\frac{d}{dt}\right\vert_{t=0} E(a + th) =
\frac12\tau\bigl(\delta_1(a)\delta_1(h) + \delta_1(h)\delta_1(a) +
\delta_2(a)\delta_2(h) + \delta_2(h)\delta_2(a)\bigr)\,.
\]
Because of the trace property, we can write this as
$\tau(\delta_1(a)\delta_1(h) + \delta_2(a)\delta_2(h))$.
For $a$ to be a critical point of $E$, this must vanish for all
choices of $h$. ``Integrating by parts'' using Lemma \ref{lem:intbyparts},
we obtain $\tau(h\Delta(a))=0$ for all $h$, and since the trace
pairing is nondegenerate, we get the {\EL} equation $\Delta a = 0$.
Since $\Delta$ has pure point spectrum with eigenvalues $-4\pi^2
(m^2 + n^2)$ and eigenfunctions $U^mV^n$, the equation has the unique
solution $a=\lambda 1$, $\lambda\in \bR$. These are also the points
where $E$ takes its minimum value of $0$.
\end{proof}

\section{Sobolev spaces}
\label{sec:Sobolev}

In the treatment of Laplace's equation above, we alluded to the theory
of Sobolev spaces. One can develop this theory in the noncommutative setting
in complete analogy with the classical case. To simplify the
treatment, we deal here only with the $L^2$ theory, which gives rise
to Hilbert spaces. These spaces are convenient for applications to
nonlinear elliptic PDE, as we will see in the next section.

\begin{definition}
\label{def:Sobolev}
For $a\in A_\theta$, we define its ``$L^2$ norm''\footnote{This is 
  really the norm for the Hilbert 
  space of the II$_1$ factor representation of $A_\theta$ determined
  by the trace $\tau$.} by
\[
\Vert a\Vert_{\ell^2} = \tau(a^*a)^{\frac12}\, .
\]
We let $L^2$ or $H^0$ (this is the Sobolev space of ``functions'' with
$0$ derivatives in $L^2$) be the completion of $A_\theta$ in this
norm. Obviously this is a Hilbert space, with inner product extending
\[
\langle a,\,b\rangle = \tau(b^* a)
\]
on $A_\theta$. Also note that the norm of $L^2$ is simply the $\ell^2$
norm for the Fourier coefficients, since if $a\in A_\theta^\infty$ has
the Fourier expansion $\sum_{m,n} c_{m,n} U^m V^n$, then
\[
\begin{aligned}
\Vert a\Vert^2_{\ell^2} &= \tau(a^*a)\\
&= \tau \left( \sum_{k,l,m,n} \left( c_{m,n} U^m V^n \right)^* 
c_{k,l} U^k V^l \right)\\
&= \tau \left( \sum_{k,l,m,n} \overline{c_{m,n}}\,c_{k,l} 
V^{-n} U^{-m} U^k V^l \right)\\
&= \sum_{m,n} |c_{m,n} |^2\, .
\end{aligned}
\]

Now let $n\in \bN$. We define the Sobolev space\footnote{Usually this
would be called $H^{n,2}$, but we are 
trying to simplify notation.} $H^n$ of ``functions'' with $n$
derivatives in $L^2$ to be the completion of $A_\theta^\infty$ in the
norm
\[
\Vert a\Vert^2_{H^n} = \sum_{0\le |\beta| \le n} \Vert
\delta_\beta(a)\Vert^2_{\ell^2}  \, .
\]
(These spaces are also defined, with slightly different notation, 
in \cite[\S3]{MR2247860}.)
Here $\beta=\beta_1\beta_2\cdots \beta_{|\beta|}$ runs over sequences
with $\beta_j = 1$ or $2$ and $\delta_\beta$ means
$\delta_{\beta_1}\cdots \delta_{\beta_{|\beta|}}$, a ``partial
derivative'' of order $|\beta|$. For example,
\[
\Vert a\Vert^2_{H^1} = \Vert a\Vert^2_{\ell^2} + \Vert
\delta_1(a)\Vert^2_{\ell^2} + \Vert \delta_2(a)\Vert^2_{\ell^2} 
\, .
\]
The Sobolev space $H^n$ is clearly a Hilbert space, and we obviously
have norm-decreasing inclusions $H^n \hookrightarrow H^{n-1}$.
Furthermore, it is clear that the Sobolev norms are invariant under
taking adjoints and can easily be
expressed in terms of the Fourier coefficients; for example, if
$a\in A_\theta^\infty$ has
the Fourier expansion $\sum_{m,n} c_{m,n} U^m V^n$, then
\[
\Vert a\Vert^2_{H^1} = \sum_{m,n} \bigl(1+4\pi^2(m^2+n^2)\bigr) 
|c_{m,n}|^2 \, .
\]
\end{definition}

The following is the exact analogue of the classical Sobolev Embedding
Theorem \cite[Theorem 1.1]{Kazdan} for $\bT^2$.

\begin{theorem}[``Sobolev Embedding'']
\label{thm:SobolevEmbedding}
The inclusion $H^n \hookrightarrow H^{n-1}$ is compact.
The space $H^1$ is not contained in $A_\theta$, but $H^2$
has a compact inclusion into $\cA_\theta$ {\lp}and thus into
$A_\theta${\rp}.
\end{theorem}
\begin{proof}
Since the Sobolev norms just depend on the decay of the Fourier
coefficients, this follows immediately from the classical 
Sobolev Embedding Theorem in dimension $2$. The inclusion of $H^2$
into $\cA_\theta$ also follows from the estimate
\[
\Vert f\Vert_{\ell^1} \le  C\Vert (-1+\Delta)f\Vert_{\ell^2},
\]
in the proof of Lemma \ref{lem:L1fromD}, with the compactness coming
from the fact that we can approximate by the finite rank operators
that truncate the Fourier series after finitely many terms.
\end{proof}

\section{Nonlinear problems involving the Laplacian}
\label{sec:nonlin}

Somewhat more interesting, and certainly more difficult to treat than
the situation of Proposition \ref{prop:classEL}, are certain
nonlinear problems involving the Laplacian, of the general form
$\Delta u = f(u)$. Such problems arise classically from the problem of
prescribing the scalar curvature of a metric $e^ug$ obtained by
conformally deforming the original metric $g$ on a Riemannian manifold
$M$ \cite[Chs.\ 5, 7]{Kazdan}. For example, if $g$ is the usual flat
metric on $\bT^2$, then the scalar curvature $h$ of the pointwise conformal
metric $e^ug$ solves the equation $\Delta u = - h e^u$. (This
equation is studied in detail in \cite[\S 5]{MR0343205}.)
Because of the
Gauss-Bonnet theorem on the torus, $h$ must integrate out to $0$, so
there are no solutions with $h$ a constant unless $h=0$ and $u$ is a
constant. This fact has an exact analogue in our noncommutative
setting.
\begin{proposition}
\label{prop:NCGaussBonnet}
If $\lambda\in \bR$, the equation $\Delta u = - \lambda e^u$ has no solution
$u = u^* \in A_\theta^\infty$ unless $\lambda = 0$ and $u$ is a scalar
multiple of $1$.
\end{proposition}
\begin{proof}
Suppose $u = u^* \in A_\theta^\infty$. Then $e^u \ge 0$, so if
$\lambda\ne 0$, either $\lambda e^u \ge 0$ or $-\lambda e^u \ge 0$.
Thus if $\Delta u = - \lambda e^u$, either $u$ or $-u$ is
subharmonic. The result now follows from Corollary \ref{cor:nosubh}.
\end{proof}
\begin{proof}[Alternative Proof]
Use the Maximum Principle, Proposition \ref{prop:MaxPrin}. Let $[a,b]$ be 
the smallest closed interval containing
the spectrum of $u$. Then for any state $\varphi$ of $A_\theta$,
$a\le \varphi(u)\le b$ and $\varphi(e^u)\ge e^a > 0$. If
$\Delta u = - \lambda e^u$ and $\lambda>0$, then by Proposition 
\ref{prop:MaxPrin}, there is a state $\varphi$ with $\varphi(u) = a$
and $\varphi(\Delta u)\ge 0$, while $\varphi(-\lambda e^u) < 0$, a
contradiction. Similarly, if $\lambda < 0$ and $\Delta u = - \lambda e^u$,
there is a state $\varphi$ with $\varphi(u) = b$
and $\varphi(\Delta u)\le 0$, while $\varphi(-\lambda e^u) > 0$, a
contradiction. 
\end{proof}

Proposition \ref{prop:NCGaussBonnet} suggests that we consider the
equation $\Delta u = - \frac12\bigl( h e^u + e^u h\bigr)$ 
with $h=h^*$ not a scalar. (Note
that we have symmetrized the right-hand side to make it self-adjoint,
since $u=u^*$ implies $\Delta u$ is self-adjoint.) Once again, a slight
variation on the argument of Proposition \ref{prop:NCGaussBonnet}
shows that there is no solution if $h\ge 0 $ or if $h\le 0$; again
this is not surprising since one gets the same result in the classical
case as a consequence of Gauss-Bonnet.
\begin{proposition}
\label{prop:NCGaussBonnet1}
If $h\ge 0$ or $h\le 0$ in $A_\theta^\infty$, the equation $\Delta u =
-\frac12\bigl( h e^u + e^u h\bigr)$ has no solution 
$u = u^* \in A_\theta^\infty$ unless $h = 0$ and $u$ is a scalar
multiple of $1$.
\end{proposition}
\begin{proof}
This is just like the proof of Proposition \ref{prop:NCGaussBonnet}.
If $h\ge 0$ and $\Delta u =- \frac12\bigl(
h e^u + e^u h\bigr)$, then applying $\tau $ to both sides, we get
\begin{equation}
\label{eq:van}
0 = \tau(\Delta u) = - \tau (h e^u) = - \tau \left(h^{\frac12} e^u
h^{\frac12} \right).
\end{equation}
Since 
\[
h^{\frac12} e^u h^{\frac12} = \left( e^{\frac u2}h^{\frac12} \right)^*
\left( e^{\frac u2}h^{\frac12} \right) \ge 0
\]
and $\tau$ is faithful, that implies $e^{\frac u2}h^{\frac12} = 0$.
Since $e^{\frac u2}$ is invertible, it follows that $h^{\frac12} = 0$
and $h=0$. The case where $h\le 0$ is almost identical; just replace
$h$ by $-h$ and change the sign of the right-hand side of \eqref{eq:van}.
\end{proof}

Unfortunately, the rest of the treatment in \cite[\S 5]{MR0343205}
doesn't extend to our setting, since from the calculation
\[
\tau(h) = \frac12 \tau\bigl(e^{-u}he^u + h \bigr) 
= -\tau(e^{-u}\Delta u),
\]
it is not clear if $\tau(h) < 0 $ follows. (The problem is
that we can't commute the various factors that arise from expanding
$\delta_j(e^{-u})$ after ``integration by parts.'') But since the main purpose
of this section is just to test various techniques and see to what
extent they apply to nonlinear noncommutative elliptic PDEs, we will
consider instead the following more tractable equation from
\cite[Ch.\ 5]{Kazdan}:
\begin{equation}
\label{eq:Kaz}
\Delta u =  \mu\,e^u - \lambda,\quad \lambda,\,\mu \in \bR,
\ \lambda,\,\mu > 0.
\end{equation}
\begin{theorem}
\label{thm:nonlin}
The equation \textup{\eqref{eq:Kaz}} has the unique solution 
$t_0=\ln(\lambda/\mu)$ in
$\bigl(A_\theta^\infty\bigr)_{\text{s.a.}}$.
\end{theorem}
\begin{proof}
Let 
\[
\cL(u) = E(u) + \tau(\mu\, e^u - \lambda u)\,.
\]
Note that for $t\in \bR$, $\mu\, e^t - \lambda t $
has an absolute minimum at $t=t_0$, so $\mu\, e^u - \lambda u 
\ge \lambda(1-t_0)$ for
$u=u^*$ and so $\cL(u)\ge \lambda(1-t_0)$ for $u=u^*$. Furthermore, the {\EL}
equation for a critical point of $\cL$ is precisely \eqref{eq:Kaz},
since
\[
\left.\frac{d}{dt}\right\vert_{t=0} \cL(u + th)
=\tau(\delta_1(u)\delta_1(h) + \delta_2(u)\delta_2(h) - \lambda h)
+\left.\frac{d}{dt}\right\vert_{t=0}\tau\bigl(\mu\,e^{u+th}\bigr)\,,
\]
via the calculation in the proof of Proposition \ref{prop:classEL}.
Now
\[
\begin{aligned}
\left.\frac{d}{dt}\right\vert_{t=0}\tau\bigl(e^{u+th}\bigr)
&=
\left.\frac{d}{dt}\right\vert_{t=0}
\sum_{n=0}^\infty \frac{1}{n!}\tau\bigl((u+th)^n\bigr)\\
&=\sum_{n=0}^\infty \frac{1}{n!}\tau\bigl(u^{n-1}h + u^{n-2}hu +
\cdots + uhu^{n-2} + h u^{n-1}\bigr)\\
&=\sum_{n=0}^\infty \frac{n}{n!}\tau\bigl(h u^{n-1}\bigr)
= \tau(h e^u)
\end{aligned}
\]
by invariance of the trace under cyclic permutations of the factors.
So applying Lemma \ref{lem:intbyparts}, we see
that
\[
\left.\frac{d}{dt}\right\vert_{t=0} \cL(u + th)
=\tau(-h\Delta(u) - \lambda h + \mu\,h e^u) = 
-\tau\bigl(h\cdot (\Delta u + \lambda - \mu\,e^u)
\bigr)\,.
\]
So nondegeneracy of the trace pairing gives \eqref{eq:Kaz}
as the {\EL} equation for a critical point of $\cL$. It is also
clear that $t_0$ is an absolute minimum for $\cL$ and a solution
of \eqref{eq:Kaz}. It remains to prove the uniqueness. Suppose
$u$ is a solution of \eqref{eq:Kaz} and write $u=t_0+v$. Then $v$ satisfies
the equation $\Delta v = \lambda(e^v-1)$, and we need to show $v=0$.
Multiply both sides by $v$ and apply $\tau$. We obtain
(using Lemma \ref{lem:intbyparts})
\[
-2 E(v) = \tau(v\Delta v) = \lambda \tau(v(e^v-1)).
\]
The left-hand side is $\le 0$, while since $\lambda>0$ and
$t(e^t-1)\ge 0$ with equality only at $t=0$, the right-hand side is
$\ge 0$. Thus $E(v)=0$, which implies 
$v$ is a scalar with $v(e^v-1)=0$, i.e., $v=0$.
\end{proof}

With techniques reminiscent of \cite[Ch.\ 5]{Kazdan}
we can study a slightly more complicated variant of \eqref{eq:Kaz}.
\begin{theorem}
\label{thm:nonlin1}
Let $a\ge 0$ be invertible in $A_\theta^\infty$. Then the
equation 
\begin{equation}
\label{eq:Kaz1}
\Delta u =  \mu\,e^u - a,\quad \mu \in \bR,\ \mu > 0
\end{equation}
has a solution $u\in \bigl(A_\theta^\infty\bigr)_{\text{s.a.}}$.
\end{theorem}

Without loss of generality (as a result of replacing $u$
by $u-\ln\mu$) we can take $\mu=1$; that simplifies the
calculations and we make this simplification from now on.
Some condition on $a$ beyond the fact that $a\ge 0$, for example
at least $a\ne 0$, is necessary because of Proposition
\ref{prop:NCGaussBonnet}, and we see that any solution of
\eqref{eq:Kaz1} must satisfy $\tau(e^u) = \tau(a) > 0$.
\begin{proof}
Several methods are available for proving existence, 
but the simplest seems to be
to apply the Leray-Schauder Theorem (\cite{MR1509338}, \cite[Theorem 
5.5]{Kazdan}). Consider the family of equations
\begin{equation}
\label{eq:Kaz2}
\Delta u =  (1-t)\,u + t \,e^u - a,\quad 0\le t\le 1 \, .
\end{equation}
When $t=0$ this reduces to $\Delta u = u - a$, or
$(-\Delta + 1)\,u = a$, which by Proposition \ref{prop:invLap} has
the unique solution $u = (-\Delta + 1)^{-1}a$. When $t=1$,
\eqref{eq:Kaz2} reduces to \eqref{eq:Kaz1}. 
We begin by using the Maximum Principle, 
Proposition \ref{prop:MaxPrin}, which implies an \emph{a priori} bound on
solutions of \eqref{eq:Kaz2}. (Compare the argument in
\cite[pp.\ 56--57]{Kazdan}.) Indeed, suppose $u$ satisfies
\eqref{eq:Kaz2} for some $0\le t\le 1$,
and let $[c,d]$ be the smallest closed interval containing
$\sigma(u)$. We may choose a state $\varphi$ of $A_\theta$ with
$\varphi(u)=d$, $\varphi(e^u)=e^d$, and by Proposition \ref{prop:MaxPrin},
$\varphi(\Delta u) \le 0$. Since 
\[
\varphi\bigl( (1-t)\,u + t \,e^u-a \bigr) =  (1-t)\,d + t \,e^d - \varphi(a)
\ge (1-t)\,d + t \,e^d - \Vert a \Vert, 
\]
we get a contradiction if $(1-t)\,d + t \,e^d - \Vert a \Vert
> 0$, which is the case if $d > \Vert a \Vert$. 
So $d\le \Vert a \Vert$.
Similarly, we may choose a state $\psi$ of $A_\theta$ with
$\psi(u)=c$, $\psi(e^u)=e^c$, and by Proposition \ref{prop:MaxPrin},
$\psi(\Delta u) \ge 0$. Since 
\[
\psi\bigl((1-t)\,u + t \,e^u - a \bigr) = (1-t)\,c + t \,e^c - \psi(a)
\le (1-t)\,c + t \,e^c - \frac{1}{\Vert a^{-1} \Vert}, 
\]
we get a contradiction 
if $e^c - \frac{1}{\Vert a^{-1} \Vert}
< 0$. Thus $e^c - \frac{1}{\Vert a^{-1} \Vert}\ge 0 $ and 
$c\ge -\ln \Vert a^{-1} \Vert$. In other words, we have
shown that any solution of \eqref{eq:Kaz2}, for any $0\le t\le 1$, 
satisfies the \emph{a priori} estimate
\begin{equation}
\label{eq:apriori}
-\ln \Vert a^{-1} \Vert \le u \le \Vert a \Vert \, .
\end{equation}

Now rewrite \eqref{eq:Kaz2} in the form
\[
u = (-\Delta + 1)^{-1}\bigl(a + t\,u - t\,e^u\bigr) \, .
\]
The right-hand side is well-defined and continuous in the
{\Ca} norm topology for $u = (A_\theta)_{\text{s.a.}}$, since
$(-\Delta + 1)^{-1}$ is bounded by Lemma \ref{lem:L1fromD}.
In fact, this Lemma also shows $(-\Delta + 1)^{-1}$ is bounded 
as a map $A_\theta \to \cA_\theta$, so as a map
$A_\theta \to A_\theta$, it is a limit of operators of finite rank, namely
the restrictions of the operator to the span of $\{U^mV^n : m^2+n^2\le N\}$,
as $N\to \infty$.  Thus $(-\Delta + 1)^{-1}$ is not only
bounded, but also compact. Together with the \emph{a priori} estimate
\eqref{eq:apriori} and the fact that there is a solution for $t=0$, 
this shows that \eqref{eq:Kaz2} satisfies the hypotheses of the 
Leray-Schauder Theorem. Hence \eqref{eq:Kaz2} has a solution for
all $t\in [0,\,1]$. Thus \eqref{eq:Kaz1} (which is the special
case of \eqref{eq:Kaz2} for $t=1$) has a solution in $\dom \Delta
\subseteq A_\theta$, and thus in $\cA_\theta$ by Lemma \ref{lem:L1fromD}.

The last step of the proof is ``elliptic regularity.'' In other words,
we need to show that a solution to \eqref{eq:Kaz1}, so far only known
to be in $\cA_\theta$, lies in $A_\theta^\infty$. Since $a\in
A_\theta^\infty$ and $\cA_\theta$ is closed under holomorphic functional
calculus (by Theorem \ref{thm:Wiener}), the right-hand side of \eqref{eq:Kaz1}
lies in $\cA_\theta$, i.e., has absolutely summable Fourier 
coefficients. Then \eqref{eq:Kaz1} implies that the Fourier coefficients
$c_{m,n}$ of $u$ have even faster decay, namely,
\[
\sum_{m,n} (1+m^2 + n^2) |c_{m,n}| < \infty \, .
\]

Now one can iterate this argument. This is a bit tricky, as at each
step one needs a new Banach subalgebra of $A_\theta$ to replace
$\cA$ (we drop the subscript $\theta$ for simplicity of notation), 
so we indicate how this works at the next step, and then
sketch how to proceed further. For $u\in \cA$ with Fourier
coefficients $c_{m,n}$, let 
\[
\Vert u\Vert_1 = \sum_{m,n} (2+m^2 + n^2) |c_{m,n}|\, ,
\]
assuming this converges. We have seen that we know $\Vert u\Vert_1
< \infty$. We claim that $\Vert \,\cdot\,\Vert_1$ is a Banach
$*$-algebra norm. This will follow by the argument in the proof of
Theorem \ref{thm:Wiener} if we can show that
\begin{multline*}
\sum_{p,q} (2 + p^2 + q^2) \sum_{m,n} |c_{m,n}|\, |d_{p-m,q-n}|
\le \\ \left( \sum_{m,n} (2+m^2 + n^2) |c_{m,n}| \right)
\left( \sum_{l,k} (2+l^2 + k^2) |d_{l,k}| \right)\, .
\end{multline*}
Comparing the two sides of this inequality, one sees it is equivalent
to proving that
\[
(2 + p^2 + q^2) \le (2+m^2 + n^2) (2+(p-m)^2 + (q-n)^2) \,,
\]
or with $\overrightarrow v = (m,n)$ and $\overrightarrow w = (p-m,q-n)$
vectors in Euclidean $2$-space, that
\[
\bigl( 2 + \Vert \overrightarrow v + \overrightarrow w \Vert^2 \bigr)
\le \bigl( 2 + \Vert \overrightarrow v \Vert^2 \bigr)
\bigl( 2 + \Vert \overrightarrow w \Vert^2 \bigr) \, .
\]
This inequality in turn follows from the standard inequality
\[
\Vert \overrightarrow v + \overrightarrow w \Vert^2 \le
\Vert \overrightarrow v  \Vert^2 + \Vert 
\overrightarrow w \Vert^2 + 2 \Vert \overrightarrow v  \Vert\cdot \Vert 
\overrightarrow w \Vert \le
2\bigl( \Vert \overrightarrow v  \Vert^2 + \Vert 
\overrightarrow w \Vert^2 \bigr)\, .
\]
This shows the completion of $A_\theta^\infty$ in the norm 
$\Vert \,\cdot\,\Vert_1$ is a Banach $*$-algebra $\cA_1$. Since
$u$ and $a$ are in $\cA_1$, so is $e^u-a$. By \eqref{eq:Kaz1} again,
$u$ has still more rapid decay; its Fourier coefficients satisfy
\[
\sum_{m,n} (m^2+n^2)^2 |c_{m,n}| < \infty\, .
\]
Now we iterate again using still another Banach $*$-algebra $\cA_2$
with the norm
\[
\Vert u\Vert_2 = \sum_{m,n} \bigl(8 + (m^2 + n^2)^2\bigr) |c_{m,n}|\, .
\]
Again one has to check that this is a Banach algebra norm, which will
follow from the inequalities
\[
\begin{aligned}
8 + \Vert \overrightarrow v + \overrightarrow w \Vert^4
&= 8 + \bigl(\Vert \overrightarrow v + \overrightarrow w \Vert^2\bigr)^2\\
&\le  8 + \bigl(2\bigl(\Vert \overrightarrow v\Vert^2 + 
\Vert \overrightarrow w \Vert^2\bigr)\bigr)^2\\
&\le 8 + 4 \bigl( \Vert \overrightarrow v \Vert^4 + 
\Vert \overrightarrow w \Vert^4 + 2 \Vert \overrightarrow v \Vert^2\cdot
\Vert \overrightarrow v \Vert^2\bigr)\\
&\le  8 + 4 \bigl( 2 \bigl(\Vert \overrightarrow v \Vert^4 + 
\Vert \overrightarrow w \Vert^4\bigr)\bigr)\\
&\le  \bigl(8 + \Vert \overrightarrow v \Vert^4\bigr)
\bigl(8 + \Vert \overrightarrow w \Vert^4\bigr)\, .
\end{aligned}
\]
Thus $\cA_2$ is a Banach algebra and $e^u-a\in \cA_2$, so that
$\Delta u\in \cA_2$ and
the Fourier coefficients of $u$ decay faster than $(m^2+n^2)^3$,
etc. Repeating in this way, we show by induction that $c_{m,n}$ is rapidly
decreasing, and thus that $u\in A_\theta^\infty$.
\end{proof}
\begin{proof}[Sketch of a second proof]
One could also approach this problem using ``variational methods.''
By the argument at the beginning of the proof of Theorem 
\ref{thm:nonlin}, \eqref{eq:Kaz1} is the {\EL} equation for critical
points of
\[
\cL(u) = E(u) + \tau(e^u - u\,a)
= E(u) + \tau(e^u - a^{\frac12} u a^{\frac12})\,.
\]
This functional is bounded below since $E(u)\ge 0$ and
$\tau(e^u  - a^{\frac12} u a^{\frac12})$ is bounded below
(by a constant depending only on $a$). Indeed, 
for $t$  and $\lambda > 0$ real, $e^t-\lambda t$ has a global
minimum at $t=\ln\lambda$, so  $e^t-\lambda t\ge \lambda(1 - \ln
\lambda)$.  If we write $u=u_+-u_-$ with
$u_+ u_- = u_- u_+ = 0$ and $u_+,\, u_- \ge 0$, then
\[ 
\begin{aligned}
-\tau(u\,a) &= \tau(u_- a)-\tau(u_+ a)\\
&= -\tau\bigl(u_+^{\frac12}au_+^{\frac12}\bigr)
 + \tau\bigl(u_-^{\frac12}au_-^{\frac12}\bigr)\\
&\ge -\tau\bigl(u_+^{\frac12} \Vert a\Vert u_+^{\frac12}\bigr) + 0\\
&= - \Vert a \Vert\, \tau(u_+)\,.
\end{aligned}
\]
On the other hand,
\[ 
\tau(e^u) = \tau\bigl(e^{u_+} + e^{-u_-} - 1\bigr)
\ge \tau\bigl(e^{u_+}\bigr) - 1,
\]
and thus
\[
\begin{aligned}
\tau(e^u  - u\,a) &\ge \tau\bigl(e^{u_+}\bigr)-  \Vert a \Vert\, \tau(u_+) 
- 1\\
&= \tau\bigl( e^{u_+} - \Vert a \Vert\,u_+ \Bigr) - 1\\ 
&\ge \Vert a \Vert\,\bigl( 1 - \ln \Vert a \Vert\bigr) - 1\,.
\end{aligned}
\]
So we will show that
$\cL$ must have a minimum point, which will be a solution of
\eqref{eq:Kaz1}.

Choose $u_n=u_n^* \in A_\theta^\infty$ with $\cL(u_n)$ decreasing to
$\inf\bigl\{\cL(u) : u\in \bigl(A_\theta^\infty\bigr)_{\text{s.a.}}
\bigr\}$. Since $E$ and $\tau(e^u  - a^{\frac12} u a^{\frac12})$ are
separately bounded below, $E(u_n)$ must remain bounded. That means that
$\Vert \delta_j(u_n)\Vert_{\ell^2}$ remains bounded for $j=1,\,2$. 

We can also assume that $\Vert u_n\Vert_{\ell^2}$ remains bounded.
To see this, it is easiest to use a trick (cf.\ \cite[pp.\ 56--57]{Kazdan}).
Because of the \emph{a priori} bound on solutions of \eqref{eq:Kaz1}
coming from the Maximum Principle (see the first proof above),
we can modify the function $e^u$ on the right-hand side of the
equation and replace it by some $C^\infty$ function that grows linearly
for $u\ge \Vert a\Vert + 1$ and decays linearly for $u\le -1
-\ln \Vert a\Vert$.  (This does not affect the Maximum
Principle argument, so the solutions of the modified equation are
the same as for the original one.)  This has the effect of changing
the term $\tau(e^u)$ in the formula for $\cL$ to something that
outside of a finite interval behaves like a constant times
$\tau(u^2)$, which is $\Vert u\Vert_{\ell^2}^2$.

Thus we can assume our minimizing sequence $u_n$ is bounded in
the Sobolev space $H^1$. Since the unit ball of a Hilbert space
is weakly compact, after passing to a subsequence, we can 
assume that $u_n$ converges
weakly in the Hilbert space $H^1$, and by Theorem
\ref{thm:SobolevEmbedding}, strongly in $H^0=L^2$,
to some $u\in H^1$ which is a minimizer for $\cL$. 
(Compare the argument in \cite[Theorem 5.2]{Kazdan}.)
This $u$ is a ``weak solution'' of our equation
and we just need to show it is smooth, i.e., corresponds to a genuine
element of $A_\theta^\infty$. This requires an ``elliptic
regularity'' argument similar to the one in the first proof.
\end{proof}

\section{Harmonic unitaries}
\label{sec:unitaries}

In this section, we discuss the noncommutative analogue of the
classical problem of studying harmonic maps $M\to S^1$, where $M$
is a compact Riemannian manifold and $S^1$ is given its usual metric.
This problem was studied and solved in \cite[pp.\
  128--129]{MR0164306}. The homotopy classes of maps $M\to S^1$ are
classified by $H^1(M,\bZ)$. For each homotopy class in $H^1(M,\bZ)$,
we can think of it as an integral class in
$H^1(M,\bR)$, and represent it (by the de Rham and Hodge Theorems) by
a unique harmonic $1$-form with integral periods. Integrating this
$1$-form gives a harmonic map $M\to S^1$ in the given homotopy
class.  This map is not quite unique since we can compose with an
isometry (rotation) of the circle, but except for this we have
uniqueness.  (This follows from \cite[Proposition, p.\
123]{MR0164306}.)

If we dualize a map $M\to S^1$, we obtain a unital $*$-homomorphism
$C(S^1)\to C(M)$, which since $C(S^1)$ is the universal {\Ca} on a
single unitary generator, is basically the same as a choice of a
unitary element $u\in C(M)$. This analysis suggests that the
noncommutative analogue of a harmonic map to $S^1$
should be a ``harmonic'' unitary in a
noncommutative  {\Ca} $A$. Each unitary in $A$ defines a class in the
topological $K$-theory group $K_1(A)$, and for $A$ a unital {\Ca},
every $K_1$ class is represented by a unitary in $M_n(A)$ for some
$n$, so since we can replace $A$ by $M_n(A)$, the natural problem is
to search for a harmonic representative in a given connected component
of $U(A)$ (or, passing to the stable limit, in a given $K_1$ class). 

The next level of complexity up from the case where $A=C(M)$ is
commutative is the case where $A=C(M,M_n(\bC))$ for some $n$. In this
case, a unitary in $U(A)$ is the same thing as a map $M\to U(n)$, and
a harmonic unitary should be the same thing as a harmonic map $M\to
U(n)$. For example, suppose $M=S^3$ and $n=2$.
Since there are no maps $M\to S^1$ which are not homotopic to a
constant, it is natural to look first at smooth  maps $f\co S^3\to
U(2)$ with $\det \circ f\co S^3\to \bT$ identically equal to $1$,
i.e., to look at maps $f\co S^3\to SU(2)=S^3$, with both copies of
$S^3$ equipped with the standard  ``round'' metric. This problem is
treated in \cite[Proposition, pp.\ 129--131]{MR0164306}. For example,
the identity map $S^3\to S^3 = SU(2) \hookrightarrow U(2)$ is a
harmonic map representing the generator of $K_1(A)= K^{-1}(S^3)$.
The study of harmonic maps  in other homotopy classes, even just in
the simple case of $S^3\to S^3$, is a complicated issue
(see, e.g., \cite[Proposition, pp.\ 129--131]{MR0164306} and
\cite{MR762354}); however, this is quite tangential to the main theme
of this article, so we won't consider it further.

Instead, we consider now the notion of harmonic unitaries in the case
of $A_\theta$. Recall first that $K_1(A_\theta)\cong \bZ^2$, with $U$ and
$V$ as generators \cite[Corollary 2.5]{MR587369}, and that the
canonical map $U(A_\theta) / 
U(A_\theta)_0 \to K_1(A_\theta)$ is an isomorphism \cite{MR887221}.
\begin{definition}
\label{def:harmunitary}
If $u\in A_\theta^\infty$ is unitary, we define the \emph{energy} of
$u$ to be
\[
E(u) = \frac12 \,\tau\bigl( (\delta_1(u))^*\delta_1(u) +
(\delta_2(u))^*\delta_2(u)\bigr) \,.
\]
Obviously this is constructed so as to be $\ge 0$. This definition
also coincides with the energy defined in Definition \ref{def:energy1}, 
provided we insert the appropriate $*$'s in the latter (which we
can do without changing anything since there we were taking $u$ to 
be self-adjoint).
The unitary $u$ is called \emph{harmonic} if it is a critical point
for $E\co U(A_\theta^\infty)\to [0,\infty)$. By the discussion above,
a harmonic unitary is the noncommutative analogue of a harmonic
circle-valued function on a manifold.
\end{definition}
\begin{remark}
\label{rem:nonunique}
Note that in Definition \ref{def:harmunitary}, $E(u)$ is
invariant under multiplication of $u$ by a scalar $\lambda\in \bT$.
Thus $E$ descends to a functional on the \emph{projective} unitary
group $PU(A_\theta^\infty)$ and any sort of uniqueness result for
harmonic unitaries can only be up to multiplication of $u$ by a 
scalar $\lambda\in \bT$. This is analogous to what happens in the
case of harmonic maps $M\to \bT$, where the associated harmonic
$1$-form is unique but the map itself is only defined up to a 
``constant of integration.''
\end{remark}
\begin{theorem}
\label{thm:ELunitary}
If $u\in A_\theta^\infty$ is unitary, then $u$ is harmonic if and only
if it satisfies the {\EL} equation
\begin{equation}
\label{eq:ELunitary}
u^*(\Delta u) + (\delta_1(u))^*\,\delta_1(u) +
(\delta_2(u))^*\,\delta_2(u) = 0\, .
\end{equation}
Note that this equation is ``elliptic'' {\lp}if we drop lower-order
terms, it reduces to Laplace's equation $\Delta u = 0${\rp}, but
highly nonlinear.
\end{theorem}
\begin{proof}
First note that for $u$ unitary, since $u\,u^* = u^* u = 1$,
we have
\[
\delta_j(u)\,u^* + u\,(\delta_j(u))^* = (\delta_j(u))^*\,u
+ u^*\, \delta_j(u) = 0\, ,
\]
$j=1,\,2$. If $u$ is unitary, then any nearby
unitary is of the form $ue^{ith}$, $h=h^*$, and
\[
\begin{aligned}
\left.\frac{d}{dt}\right\vert_{t=0} E(ue^{ith})&= \frac12\,
\tau\Bigl(-i\delta_1(h)u^*\delta_1(u) + i \delta_1(u)^*u \delta_1(h)\\
&\quad + \text{similar expression with }\delta_2\Bigr)\,.
\end{aligned}
\]
We can use the trace property to move all the $\delta_j(h)$'s to
the front. So $u$ is a critical point if and only if for
all $h=h^*$,
\begin{equation}
\label{eq:ELLaplace}
\tau\Bigl(\delta_1(h)\Im\bigl(u^*\delta_1(u)\bigr)
+ \delta_2(h)\Im\bigl(u^*\delta_2(u)\bigr)\Bigr) = 0\,.
\end{equation}
In \eqref{eq:ELLaplace}, the $\Im$'s can be omitted since
we have seen that $u$ unitary $\Rightarrow$ $\delta_j(u)^*u$ skew-adjoint.
Thus $u$ is harmonic if and only if
\[
\tau\Bigl(\delta_1(h)\,\bigl(u^*\delta_1(u)\bigr)
+ \delta_2(h)\,\bigl(u^*\delta_2(u)\bigr)\Bigr) = 0
\]
for all $h=h^*$ in $A_\theta^\infty$.
Now apply integration by parts (Lemma \ref{lem:intbyparts}). We
see that $u$ is harmonic if and only if
\[
\tau\Bigl(h\,\delta_1\bigl(u^*\delta_1(u)\bigr)
+ h\,\delta_2\bigl(u^*\delta_2(u)\bigr)\Bigr) = 0
\]
for all $h=h^*$ in $A_\theta^\infty$. Since the trace pairing is
nondegenerate, the Theorem follows.
\end{proof}

It seems natural to make the following conjecture:
\begin{conjecture}
\label{conj:harmu}
In each connected component of $PU(A_\theta^\infty)$, the functional
$E$ has a unique minimum, given by scalar multiples of $U^nV^m$.
These are the only harmonic unitaries in this component.
\end{conjecture}

Unfortunately, because of the complicated nonlinearity in
\eqref{eq:ELunitary}, plus complications coming from noncommutativity,
we have not been able to prove the Conjecture
\ref{conj:harmu}. However, we have the following partial result.
In particular, we see that every connected component
in $U(A_\theta^\infty)$ contains a harmonic unitary which is
energy-minimizing.
\begin{theorem}
\label{thm:NClapAtheta}
The scalar multiples of $U^mV^n$ are harmonic and are strict 
local minima for $E$. Any harmonic unitary $u$ depending on $U$ alone
is a scalar multiple of a power of $U$. Similarly, any harmonic unitary 
$u$ depending on $V$ alone is a scalar multiple of a power of $V$. 
\end{theorem}
\begin{proof}
First suppose $u$ depends on $U$ alone. Then
$\delta_2(u)=0$. So by the proof of Theorem \ref{thm:ELunitary},
if $u$ is harmonic, then 
$\tau\Bigl(\delta_1(h)\cdot\delta_1(u)^*u\Bigr)=0$
$\forall h=h^*$. This must also hold for general $h$ (not necessarily
self-adjoint) since we can split $h$ into its self-adjoint and
skew-adjoint parts. Since the range of $\delta_1$ contains $U^m$
unless $m=0$, $\tau(\delta_1(u)^*u\,U^m)$=0 for $m\ne 0$, which
means (since $\delta_1(u)^*u$ depends only on $U$) that
$\delta_1(u)^*u$ is a scalar.  Thus $u$ is an
eigenfunction for $\delta_1$ and so $u=e^{i\lambda}U^m$ for some $m$.
The case where $u$ depends on $V$ alone is obviously similar.

Next let's examine $u=U^mV^n$. Since $E(U^mV^n) = 2\pi^2 (m^2+n^2)$
while 
\[
(U^mV^n)^*\Delta(U^mV^n) = -4\pi^2 (m^2+n^2) \, , 
\]
$u$ satisfies
\eqref{eq:ELunitary} and is therefore harmonic. We show it is a 
local minimum for $E$; in fact, the minimum is strict once we pass
to $PU(A_\theta^\infty)$.  We expand $\delta_j(ue^{ith})$, with $h=h^*$,
out to second order in $t$. Note that with $\delta=\delta_1$ or $\delta_2$,
\[
\delta(ue^{ith}) =
\delta(u) + i t\,\bigl[ \delta(u)h + u\delta(h) \bigr]
- \frac{t^2}{2}\,\bigl[ \delta(u)\,h^2 + u\,\delta(h)\,h + u\,h\,\delta(h)
\bigr] + O(t^3).
\]

We substitute this into the formula for $E(ue^{ith})$.
The terms linear in $t$ cancel since $u$ is harmonic, and we find that
\[
\begin{aligned}
E\bigl(u\,e^{ith}\bigr) &= 2\pi^2 (m^2+n^2) + 
t^2\,\tau\bigl[ (\delta_1(u)h + u\delta_1(h) )^*(\delta_1(u)h 
+ u\delta_1(h) ) \\
&\qquad
- \frac{1}{2}\delta_1(u)^*(\delta_1(u)\,h^2 + u\,\delta_1(h)\,h + 
u\,h\,\delta_1(h))\\
&\qquad
- \frac{1}{2}(h^2\,\delta_1(u)^* + h\,\delta_1(h)\,u^* + 
\delta_1(h)\,h\,u^*)\,\delta_1(u)\\
&\qquad + \text{similar expressions with }\delta_2\Bigr]+ O(t^3)\,.
\end{aligned}
\]
This actually simplifies considerably since $u$ is an eigenvector
for both $\delta_1$ and $\delta_2$, so that $\delta_j(u)^*\delta_j(u)$,
$\delta_j(u)^*u$, and $u^*\delta_j(u)$ are all scalars.
It turns out that almost everything cancels and one gets
\[
\begin{aligned}
E(u e^{ih t})&=
2\pi^2(m^2+n^2) + \frac12
t^2\tau\Bigl(\delta_1(h)^2+\delta_2(h)^2\Bigr) + O(t^3)\\
&= 2\pi^2(m^2+n^2) + {t^2}\, E(h) + O(t^3)\, .
\end{aligned}
\]
By Proposition \ref{prop:classEL}, the term in $t^2$
vanishes exactly when $h$ is a 
constant, and in that case $E(u e^{ih t}) = E(u) = 2\pi^2
(m^2+n^2)$ (exactly).
Otherwise, the coefficient of $t^2$ is strictly positive and $E(u e^{ih t})$
has a strict local minimum at $t=0$.
\end{proof}

\section{The Laplacian and holomorphic geometry}
\label{sec:holo}

As we have seen, $\Delta$ on $A_\theta$ behaves very much like the
classical Laplacian on $\bT^2$. But the Laplacian in (real) dimension $2$
is very closely related to holomorphic geometry in complex dimension
$1$. That suggests that the theory we have developed should be closely
related to the Cauchy-Riemann operators $\partial$ and $\dbar$
on ``noncommutative elliptic curves,'' as developed in references
like \cite{MR1977884,MR2054986}.

In classical analysis (in one complex variable), one usually sets
$\dbar = \frac12 \bigl(\frac{\partial}{\partial x_1} + i\,
\frac{\partial}{\partial x_2} \bigr)$, the Cauchy-Riemann operator,
with $\partial$ its complex conjugate. Then $\Delta = 4\, \partial\,
\dbar$. In our situation, the obvious analogue is to set $\dbar =
\frac12 \bigl(\delta_1 + i\,\delta_2\bigr)$.\footnote{We could also
  study different conformal structures on the torus, by changing the
  $i$ here to another complex number in the upper half-plane, but for
  the problems we will study here, this makes no essential difference.} 
Comparable to Proposition
\ref{prop:imLap} is:
\begin{proposition}
\label{prop:imdbar}
The operator $\dbar\co A_\theta^\infty\to A_\theta^\infty$ has kernel
given by scalar multiples of the identity, and restricts to a
bijection on $\ker\tau$.
\end{proposition}
\begin{proof}
Immediate from the fact that if $a=\sum_{m,n}c_{m,n}U^m V^n$, then
\[\dbar a = \pi\,i \sum_{m,n}(m+in)\,c_{m,n}U^m V^n\, ,
\]
together with the characterization of elements of $A_\theta^\infty$ in
terms of rapidly decreasing Fourier series.
\end{proof}
Thus the noncommutative torus admits no nontrivial global
``holomorphic functions.'' This is not surprising since a compact
complex manifold admits no nonconstant global
holomorphic functions.  However, assuming $\tau(f)=0$, we can solve the
\emph{inhomogeneous Cauchy-Riemann equation} $\dbar u = f$, which
in the classical case is related to the proof of the Mittag-Leffler
Theorem (see, for example, \cite[Ch.\ 1]{MR1045639}). 

In some situations, one is led to the more complicated equation
(similar to the above but with $\dbar$ replaced by the ``logarithmic
Cauchy-Riemann operator'') $(\dbar u)\,u^{-1} = f$, which we can rewrite
as $\dbar u = f u$. This equation was already studied 
(modulo a change of conventions about whether one should multiply on
the left or the right) in a (different) noncommutative context in
\cite{MR1062964}, and then in \cite{MR2247860}.
\begin{theorem}[{Polishchuk, \cite{MR2247860}}]
\label{thm:Cousin}
Let $f\in A_\theta$. Then the equation $\dbar u = fu$ has a
nonzero solution if and only if $\tau(f)\in \pi i(\bZ+i\bZ)$.
\end{theorem}
[Comment: Polishchuk and Schwarz in
\cite{MR1977884,MR2247860} use a slightly different convention.
They take $\dbar$ to be $(x+iy)\delta_1 + \delta_2$, with $y<0$;
when $x=0$ and $y=-1$, this is what we have here, up to a constant
factor of $-2i$. This constant explains why the result looks
different. With our convention, $u=U^mV^n$ solves $\dbar u = fu$
with $f = \pi i(m+i n)$.]

The relevance of this result concerns the theory of noncommutative
\emph{meromorphic} functions. While a compact
complex manifold admits no nonconstant global
holomorphic functions, it can admit nonconstant meromorphic
functions, such as (in the case of an elliptic curve) elliptic
functions like the Weierstra{\ss} $\wp$ function.
There are two ways we can view meromorphic functions on a Riemann
surface $M$. On the one hand, they can be considered as ratios of
holomorphic sections of holomorphic line bundles $\cL$ of $M$. On the
other hand, they can be considered as formal quotients of functions
that satisfy the Cauchy-Riemann equation.

These points of view, applied to a noncommutative torus, 
are equivalent via the following reasoning. A
holomorphic vector bundle is defined via its module of (smooth)
sections, which is a finitely generated projective (right)
$A_\theta^\infty$-module. This module must be equipped with an
operator $\nbar$ satisfying the basic axiom
\[
\nbar(s\cdot a) = \nbar(s)\cdot a + s\cdot \dbar(a)\,.
\]
If we assume the module is $A_\theta^\infty$ itself (i.e., the vector
bundle is of ``dimension $1$,'' i.e., is a line bundle), then this
operator is determined by $f=\nbar(1)$, in that for any $s$,
\[
\nbar(s) = \nbar(1\cdot s)
= f\cdot s + 1\,\dbar(s) = \dbar(s)+ f s\, .
\]
A ``holomorphic section'' of the bundle is then a solution $s$ of
$\dbar(s)+ f s = 0$.

On the other hand, the natural definition of ``meromorphic functions''
is the following.
\begin{definition}
\label{def:mero}
A \emph{meromorphic function on the noncommutative torus} $A_\theta$
is a formal quotient $u^{-1}v$, with $u,\,v\in \dom(\dbar) \subset
A_\theta$, satisfying the Cauchy-Riemann equation (in the
sense to be made precise below). Here we don't want to require
that $u$ be invertible in $A_\theta$ (otherwise $u^{-1}v$ would be
holomorphic, hence constant), so we simply want $u$ to be ``regular''
(in the sense of not being either a left or right zero divisor), and
the inverse is to be interpreted in a formal sense 
(or in the maximal ring of quotients \cite{MR649748}, the
algebra of unbounded operators affiliated to
the hyperfinite II$_1$ factor obtained by completing $A_\theta$ 
in its trace representation).  Then the condition that $u^{-1}v$
be meromorphic is that
\[
0 = \dbar(u^{-1}v) = \dbar(u^{-1})v + u^{-1}\dbar v
=-u^{-1}\dbar(u)u^{-1}v + u^{-1}\dbar v \, ,
\]
or (via multiplication by $u$ on the left) that
\begin{equation}
\label{eq:holopair}
\dbar v = f v,\quad \dbar u = f u \, ,
\end{equation}
which says precisely that our meromorphic function is a quotient
of two holomorphic sections of a holomorphic line bundle with
$\nbar =\dbar+f$. In the other direction, if $u$ and $v$ satisfy
\eqref{eq:holopair} and $u$ is regular, so that the formal expression
$u^{-1}v$ makes sense, then we formally have
\[
\begin{aligned}
\dbar(u^{-1}v) &= \dbar(u^{-1})v + u^{-1}\dbar v =
-u^{-1}\dbar(u)u^{-1}v + u^{-1}\dbar v \\
& = -u^{-1}f u u^{-1}v + u^{-1}fv = -u^{-1}f v + u^{-1}f v = 0\, ,
\end{aligned}
\]
and $u^{-1}v$ is meromorphic.
\end{definition}

In accordance with the classical existence theorem of Weierstra{\ss} 
for elliptic functions, we have:
\begin{proposition}
\label{prop:mero}
There exist nonconstant meromorphic functions on the noncommutative
torus $A_\theta$, in the sense of Definition \textup{\ref{def:mero}}.
\end{proposition}
\begin{proof}
This follows immediately from the discussion in \cite[\S3]{MR2247860},
which shows that there are choices for $f$ for which the holomorphic
connection $\nbar$ is reducible, with a space of holomorphic
sections of dimension bigger than $1$, and thus there are solutions of
\eqref{eq:holopair} with $u$ and $v$ not linearly dependent. Note
that if this is the case, $u$ cannot be invertible (\cite[Lemma
3.14]{MR2247860}---we also know this independently from Proposition
\ref{prop:imdbar}). But we do require $u$ to be regular, so we need
to check that this can be achieved. For example, suppose $e$ is a 
proper projection in $A_\theta^\infty$ (``proper'' means $0 < \tau(e)
< 1$) of trace $m+n\theta$ with $n$ relatively prime to both $m$ and $1-m$. 
The trivial rank-one right $A_\theta^\infty$ module splits
as $e A_\theta^\infty \oplus (1-e) A_\theta^\infty$, and we can
arrange to choose a holomorphic connection on $A_\theta^\infty$ 
that is reducible in a way compatible with this splitting, so that
there are $1$-dimensional spaces of holomorphic sections on each of
$e A_\theta^\infty$ and $(1-e) A_\theta^\infty$.  By the explicit
formulas in \cite[Proposition 2.5]{MR1977884}, these come
from real-analytic functions in $\cS(\bR)$, and so it's evident
that the $u$ that results from putting these together is regular,
as by \cite{MR649748}, it's enough to observe that its left and
right support projections are equal to $1$.
\end{proof}

On the other hand, there is also a non-existence result for
meromorphic functions on the [classical] torus: no such nonconstant
function exists 
with only a single simple pole \cite[Corollary to Theorem 4, p.\ 
271]{MR510197}. We can find an analogue of this in the
noncommutative situation also. To explain it, first note that
in the sense of distributions on the complex plane, 
$\dbar\left(\frac{1}{z}\right)$ is not
zero (if it were, $\frac{1}{z}$ would have a removable singularity,
by elliptic regularity), but rather is equal to $\pi\,\delta$,
where $\delta$ is the Dirac $\delta$-distribution at $0$.
Suppose there were a meromorphic function $f$ on $\bT^2=\bC/(\bZ + i\bZ)$
with at worst one simple pole and no other poles. Then
$f$ would be locally integrable and, after translation to
move the pole to $0$, would define a distribution
on $\bT^2$ with $\dbar(f)$ a multiple of $\delta$. Thus the Fourier
series of $\dbar(f)$ would be a multiple of the Fourier series of
$\delta$, which is $\sum_{m,n}U^mV^n$. And in fact Fourier
analysis gives
another proof of the nonexistence theorem not using residue calculus.
Suppose $f$ were nonconstant. Since a compact complex manifold
admits no nonconstant holomorphic functions, $f$ cannot be
holomorphic, which means that $\dbar f$ 
must  be non-zero in the sense of distributions. Since $\dbar(f)$ is a
multiple of $\sum_{m,n}U^mV^n$, the proportionality constant, which is
also the $(0,0)$ Fourier coefficient
of $\dbar f$, must be non-zero. But this is impossible since
the Fourier series of any distribution in the image on $\dbar$
must have zero constant term.  The noncommutative analogue of all this
is the following:
\begin{proposition}
\label{prop:nomeroonepole}
Let $f$ be a distribution in the dual of $A_\theta^\infty$.
{\lp}The distributions consist of formal Fourier series
$\sum_{m,n}c_{m,n}\,U^mV^n$ with $\{c_{m,n}\}$ of tempered growth.{\rp}
Suppose $\dbar f$ is a multiple of $\sum_{m,n}U^mV^n$. 
Then $f$ is a constant.
\end{proposition}
\begin{proof}
This follows exactly the lines as the argument above
for the classical theorem. If $\dbar f$ has formal Fourier
expansion $c \sum_{m,n}U^mV^n$, then the $(m,n)$ coefficient, $c$,
must be divisible by $m+in$ for all $(m,n)$. Because
of the $(0,0)$ coefficient, this is only possible if $c=0$. But if $c=0$,
then $f$ is in the distributional kernel of $\dbar$, which forces
all the Fourier coefficients of $f$ to vanish except for the constant term.
\end{proof}

In fact, essentially the same proof proves a slightly more general
statement, which in the classical case is equivalent to \cite[Theorem 4, p.\ 
271]{MR510197}. For the analysis above shows that the ``sum of the
residues'' of a meromorphic function $f$ on $\bT^2$, when the function is
considered as a distribution\footnote{This requires a comment. A
meromorphic function with simple poles is locally integrable, thus
defines a distribution in the obvious way. A meromorphic function with
higher-order poles is not locally integrable, but can be made into a
distribution of ``principal value integral'' type. This distribution
is not a measure.}, is precisely the constant term in the
Fourier series of $\dbar f$, up to a factor of $\pi$. The analogue of
the ``sum of the residues'' theorem in the noncommutative world is: 
\begin{proposition}
\label{prop:sumofres}
Let $f$ be a distribution in the dual of $A_\theta^\infty$.
Then the constant term in the {\lp}formal{\rp} Fourier series
of\/ $\dbar f$ is zero.
\end{proposition}
\begin{proof}
Essentially the same as before.
\end{proof}

The connection with the main subject of this paper is of course that
``meromorphic functions'' $w$ as studied in this section are
``singular'' solutions of 
Laplace's equation $\Delta w=0$, since $\Delta = 4\, \partial\,
\dbar$. More precisely, ``singular solution'' means classically that
as a distribution, 
$\Delta w$ is not necessarily $0$, but has countable support. In the
noncommutative setting, we do not have a notion of support for a
distribution, but the same basic idea applies.

\renewcommand{\MR}{\relax}
\bibliographystyle{amsplain}
\bibliography{Laplace}
\end{document}